\documentclass{amsart} 
\usepackage{amsmath,amssymb,amsthm}
\usepackage[backref=page, colorlinks, linkcolor=blue]{hyperref}
\usepackage{enumerate}

\newtheorem{theorem}{Theorem}[section]
\newtheorem{lemma}[theorem]{Lemma}
\newtheorem{corollary}[theorem]{Corollary}
\newtheorem{proposition}[theorem]{Proposition}
\newtheorem*{theorem*}{Theorem}
\newtheorem*{corollary*}{Corollary}

\theoremstyle{definition}
\newtheorem{definition}[theorem]{Definition}
\newtheorem{example}[theorem]{Example}
\newtheorem{remark}[theorem]{Remark}

\begin{document}
\title{A Poisson limit theorem for Gibbs-Markov maps}
\author{Xuan Zhang}
\address{Instituto de Matemática e Estatística, Universidade de São Paulo}
\email{xuan@ime.usp.br}
\date{}
\begin{abstract}
We prove for Gibbs-Markov maps that the number of visits to a sequence of shrinking sets with bounded cylindrical lengths converges in distribution to a Poisson law. Applying to continued fractions, this result extends Doeblin's Poisson limit theorem.
\end{abstract}
\maketitle

\section{Introduction}
A Poisson limit theorem for continued fractions was proved by Doeblin (\cite{doeblin1940}) around 1937, which says that the number of certain large partial quotients in a continued fraction expansion converges in distribution to a Poisson law. More precisely, denoting by $[a_1,a_2,\ldots], a_i\in\mathbb N,$ the continued fraction expansion of an irrational number $x\in (0,1)$,
\begin{theorem*} 
For every $\theta>0, k\in\mathbb N_0$
\begin{multline*}
\lim_{n\rightarrow\infty}{\rm Leb}\{x\in(0,1): \text{there are exactly $k\ a_i$'s with $a_i>{\theta n}$, $1\leqslant i\leqslant n$}\}\\
=\frac{1}{(\theta\log2)^k}\frac{1}{k!}e^{-\frac{1}{\theta\log2}}.
\end{multline*}
\end{theorem*}
Later Iosifescu (\cite{Iosifescu1977}) proved a Poisson limit theorem for $\psi$-mixing processes, which in particular can be applied to the process $\{a_n\}_{n\in\mathbb N}$. Treating continued fractions as a symbolic dynamical system, one can rephrase Doeblin's theorem as a Poisson limit theorem for return times under the Gauss map $T: x\mapsto \{1/x\}$.
\begin{theorem*} 
For every $\theta>0, k\in\mathbb N_0$
$$
\lim_{n\rightarrow\infty}{\rm Leb}\left\{x\in(0,1): \sum_{i=0}^{n-1}{\bf 1}_{(0,\frac{1}{[\theta n]+1})}\circ T^i(x)=k\right\}
=\frac{1}{(\theta\log2)^k}\frac{1}{k!}e^{-\frac{1}{\theta\log2}}.
$$
\end{theorem*}
Compared with the Poisson limit theorem for binomial random variables, an ideal Poisson limit theorem in dynamical systems would be the following. Let $(\Omega,\mathcal B,\mu)$ be a probability space, $T$ be a $\mu$-preserving map on $\Omega$, and $\{A_n\}_{n\in\mathbb N_0}$ be a sequence of measurable sets. If $n\mu(A_n)\to t>0$ as $n\to\infty$, then for every $k\in\mathbb N_0$ 
\begin{equation*}
\lim_{n\to\infty}\mu\left(\left\{x\in\Omega:\sum_{i=0}^{n-1}{\bf 1}_{A_n}\circ T^i(x)=k\right\}\right)= \frac{t^k}{k!}e^{-t}.
\end{equation*}
The summation inside the brackets counts the number of visits of $x$ to $A_n$ until the $n$-th iteration by $T$. Note that, since $\{\mathbf 1_{A_n}\circ T^i\}_{i\in\mathbb N_0}$ in general is not an independent process, it is not clear that the limit distribution exists, let alone being Poisson. The study of Poisson limit law for return times in dynamical systems was initiated in the early 1990s by  Hirata (\cite{Hirata1993}), Pitskel (\cite{Pitskelprime1991}), Sinai (\cite{Sinaui1990}) etc. Since then, there are various results confirming the Poisson limit theorem and also obtaining error estimates under appropriate assumptions. See for example \cite{CoelhoCollet1994, Denker1995, HirataSaussolVaienti1999, Haydn2000, Dolgopyat2004, DenkerGordinSharova2004, AbadiVergne2008, ChazottesCollet2013, PeneSaussol2016, HaydnYang2016}. But in general, the class of possible limit distributions is rather large, more than just Poisson, as shown by Lacroix (\cite{Lacroix2002}) and Chaumo\^itre and Kupsa (\cite{Chaumoitrekupsa2006}), We refer to Haydn's recent review (\cite{Haydn2013}) on this and related topics.

Let $n\mu(A_n)\to t$ as $n\to \infty$, then the measure $\mu(A_n)$ shrinks to $0$. Most of the previous results consider a sequence of shrinking sets $A_n$ which converge to a generic non-periodic point (like in Theorem \ref{thm:poissonap} below). When $A_n$ shrinks to a periodic point, $A_n$ always intersects with its pre-images via small iterations of $T$. In this case it had been pointed out by both Pitskel and Hirata that the limit distribution can not be Poisson.
Later Haydn and Vaienti (\cite{HaydnVaienti2009}) showed for some mixing maps, that when $A_n$ are cylinder sets of growing cylindrical lengths (the level of the filtration in which $A_n$ lives)  and converging to a periodic point, the limit distribution of the number of return times becomes compound Poisson. But none of these results covers Doeblin's theorem, in which case $A_n=(0, \frac1{[\theta n]+1})$ 
converges to a non-generic point, the endpoint $0$, and has a constant cylindrical length. 

In this note we study the limit distribution of the number of visits to a sequence of shrinking sets with bounded cylindrical lengths. When the sequence converges to a single point, the accumulation point may be periodic or a compactification point of the space (like the endpoint $0$ in Doeblin's theorem). 
This type of sequences seems novel in the literature to our knowledge. Our main result (Theorem \ref{thm:main}) shows a Poisson limit theorem for Gibbs-Markov maps with this type of sequences. Some applications to continued fractions are given at the end of the note, including an extension of Doeblin's theorem.

Examples of Gibbs-Markov maps are given in the next section, including some Markov chains with countable states, Markov maps of the unit interval, parabolic rational maps. 
We use perturbation of transfer operator to study limit distribution of the number of return times, following the idea of Hirata (\cite{Hirata1993}). It turns out that for the type of sequences in our consideration, the size of this perturbation  is not asymptotically small in the Banach space where the transfer operator has good spectral property. Still, a perturbation theorem of Keller and Liverani (\cite{KellerLiverani1999}) allows us to deal with this situation using a weaker norm. We note that this perturbation theorem has been used widely by many authors to study closely-related problems, such as limit distribution of the first return time and escape rate, for example in \cite{KellerLiverani2009, FergusonPollicott2012, DemersWright2012, BruinDemersTodd2018}.

\section{Gibbs-Markov maps}
 \label{sec:GibbsMarkov}
We recall the definition of Gibbs-Markov maps from \cite{AaronsonDenker2001}. Let $(\Omega,\mathcal B, \mu)$ be a probability
space and let $T$ be a non-singular transformation. Consider a countable partition $\alpha$ of $\Omega\mod \mu$, let $\alpha=\{a_i: i\in I\}$. Denote by $\alpha_0^{n-1}$ the refined partition $\bigvee_{i=0}^{n-1}T^{-i}\alpha$ and by $\sigma(\cdot)$ the $\sigma$-algebra generated by a partition. For a set $A\in \sigma(\alpha_0^{n-1})$, we call the minimal $m\in\mathbb N$ such that $A\in \sigma(\alpha_0^{m-1})$ its cylindrical length.
\begin{definition}\label{def:gibbsmarkov}A quintuple $(\Omega,\mathcal{B},\mu,T,\alpha)$ is called a Gibbs-Markov map if it satisfies the following conditions. 
\begin{enumerate}
	\item $\alpha$ is a strong generator of $\mathcal{B}$ under $T$, i.e. $\sigma(\{T^{-n}\alpha:n\in\mathbb N_0\})=\mathcal{B} \mod \mu$.
	\item (big image property) $\inf\limits_{a\in\alpha}\mu(Ta)>0$.
	\item For every $a\in\alpha$, $Ta\in \sigma(\alpha) \mod \mu$, moreover the restriction $T|_{a}$ is invertible and non-singular. 
	\item For every $n\in\mathbb N$ and $a\in\alpha_0^{n-1}$, denote the non-singular inverse branch of $T^{-n}$ on $T^n a$ by $v_a: T^n a\rightarrow a$ and its Radon-Nikodym derivative by $v'_a$. There exist $r\in (0,1)$ and $M>0$ such that for any $n\in \mathbb N, a\in\alpha_0^{n-1}$ and $x,y\in T^n a$ a.e.
			$$\left|\dfrac{v'_a(x)}{v'_a(y)}-1\right|\leqslant M\cdot r(x,y),
			$$
\end{enumerate}
where $r(x,y)$ is the metric $r^{\min\{n\in\mathbb N: T^{n-1}(x) \text{ and } T^{n-1}(y) \text{ belong to different elements of } \alpha \}}.$ 
\end{definition}
We present some typical examples. The first two have already been mentioned in \cite{AaronsonDenker2001}. 

\begin{example}[Markov chain with countable states]\label{ex:mc} Let $S$ be a countable set, $P=(p_{x,y})_{x,y\in S}$ be an irreducible ergodic stochastic matrix on $S$ and $\pi$ be the stationary probability measure. Let $\Omega=\{x=(x_0,x_1,\ldots)\in S^{\mathbb{N}_0}:p_{x_k x_{k+1}}>0, k\in\mathbb N_0\}$, $\mathcal{B}$ the $\sigma$-field generated by all the cylinder sets of the form $[s_0, \cdots, s_n]=\{x\in\Omega: x_0=s_0,\ldots, x_n=s_n\}$, $T$ be the shift map, $\alpha=\{[s]:s\in S\}$, $\mu([s_0, \cdots, s_n])=\pi_{s_0}\prod\limits_{k=0}^{n-1}p_{s_k s_{k+1}}$. 

Suppose $a=[s_0, \cdots, s_n]$, then $v'_a(x)=\frac{\pi_{s_0}p_{s_0s_1}\cdots p_{s_{n-1}s_{n}}p_{s_{n}x_0}}{\pi_{x_0}}$ for $x\in T^na$. The system is Gibbs-Markov if and only if $\left|{\frac{p_{sx}}{\pi_x}}/{\frac{p_{sy}}{\pi_y}}-1\right|$ is uniformly bounded for all $s,x,y\in S$ with $p_{sx},p_{sy}>0$.

An explicit example is given by (\cite[XV($2.k$)]{Feller1968}):
\begin{equation*}
P=\begin{bmatrix}
	f_1 & f_2 & f_3 & f_4 &\cdots \\
	1 & 0 & 0 & 0 &\cdots \\
	0 & 1 & 0 & 0&\cdots \\
	0 & 0 & 1 & 0&\cdots \\
	\cdot & \cdot & \cdot & \cdot &\cdots\\
	\cdot & \cdot & \cdot & \cdot &\cdots
\end{bmatrix}
\end{equation*}
where $\sum_{i\in\mathbb N} f_i=1$. Let $r_k=\sum_{i>k}f_i$. Suppose $\sum_{k\in\mathbb N} r_k<\infty$, then the stationary measure $\pi=(\pi_0, \pi_0 r_1, \pi_0 r_2,\ldots)$. $T$ is Gibbs-Markov if and only if $C^{-1}\leqslant {\frac{f_i}{r_{i-1}}}/{\frac{f_j}{r_{j-1}}}\leqslant C$ for some constant $C>0$ and for all $i,j$. For example, $f_i=\frac{1}{i!}\frac{1}{e-1}$.
\end{example}

\begin{example}[Markov interval map, {\cite[\S 7.4]{CornfeldFominSinaui1982}}]\label{ex:mip}
Let $\Omega=(0,1)$, $\mathcal{B}$ be Borel $\sigma$-algebra, $\rm Leb$ be Lebesgue measure, and $\alpha=\{a_i\}$ be a partition of $\Omega \mod \mu$ into open intervals. Suppose that $T:\Omega\rightarrow \Omega$ satisfies the following conditions.
\begin{itemize}
\item $T|_{a_i}$ is strictly monotonic and extends to a $C^2$ function on $\overline{a}_i$ for each $i$.
\item if $T(a_k)\cap a_j\neq\emptyset$, then $T(a_k)\supset a_j$.
\item $T$ is expanding: there exists an $m\in\mathbb{N}$ such that
$$
\inf_{a_i\in\alpha}\inf_{x\in a_i} |T'(x)|>0,\quad
\inf_{a_i\in\alpha}\inf_{x\in a_i} |{T^m}'(x)|:=\lambda>1.
$$
\item $T$ satisfies the Adler property:
\begin{equation*}
\sup_{a_i\in\alpha}\sup_{x,y\in a_i}\frac{|T''(x)|}{T'(y)^2}<\infty.
\end{equation*}
\end{itemize}

Then there exists a constant $M$ such that:
\begin{equation*}
\frac{|{T^n}''(x)|}{{T^n}'(y)^2}\leqslant M, \text{ for all } x,y\in a\in\alpha_0^{n-1}, n\in\mathbb N.
\end{equation*}
Hence for some constant $M_1$, and for all  $n\in\mathbb N, a\in\alpha_0^{n-1}, x,y\in T^n a$,
\begin{equation*}
\left|\frac{v'_a(x)}{v'_a(y)}-1\right|\leqslant M_1 |x-y|.
\end{equation*}
Let $r=(1/\lambda)^{\frac1m}$. There is a constant $M_2$ such that for any $x,y$, if $k\in\mathbb N$ is the minimum natural number such that $T^{k-1}x$ and $T^{k-1}y$ belong to different elements of $\alpha$, supposing $x, y\in b\in \alpha_0^{k-1}$,
\begin{equation*}
|x-y|\leqslant {\rm Leb}(b)={\rm Leb} (v_b(T^k b))\leqslant \|v'_b\|_\infty\leqslant M_2 r^k.
\end{equation*}
The big image property may not always hold for $T$. If it does, then $T$ is a Gibbs-Markov map. For example, let $\alpha=\{\{a_1=n\}: n\in\mathbb N\}$ be the partition by the first partial quotient in the continued fraction expansion and let $T$ be the Gauss map: $x\mapsto \{\frac{1}{x}\}$. Then it is a Gibbs-Markov map. Furthermore, let $\mu$ be the Gauss measure $\mu(A)=\frac{1}{\log2}\int_A \frac{1}{1+x} dx$, then the continued fraction map $((0,1), \mathcal B, \mu, T, \alpha)$ is a probability-preserving Gibbs-Markov map.

The following example appears in \cite{AaronsonDenker1999} in the study of Poincar\'e series of the surface $\mathbb C\setminus \mathbb Z$.
Let 
\begin{equation*}
R(x)=\begin{cases}
\frac{x}{1-2x}, & x\in(0,\frac13)\\
\frac{1}{x}-2, & x\in(\frac13,\frac12)\\
2-\frac{1}{x}, & x\in(\frac12,1)
\end{cases}.
\end{equation*}
Denote $A:=(0,\frac13), B:=(\frac13,\frac12), C:=(\frac12,1)$, and $U:=A\cap R^{-1}A^c=(\frac15,\frac13), W:=C\cap R^{-1}C^c=(\frac12,\frac23), J:=U\cup B\cup W=(\frac15, \frac23).$
Induce $R$ on $J$, that is, let $$R_J(x)=R^{\tau}(x)$$ where $\tau=\min\{k\in\mathbb N: R^k(x)\in J\}$. Define $ B_1=B\cap R^{-1}J$, 
\begin{gather*}
U_n=U\cap\bigcap_{j=1}^{n-1}R^{-j}C\cap R^{-n}(B\cup W),\quad W_n=W\cap\bigcap_{j=1}^{n-1}R^{-j}A\cap R^{-n}(U\cup B),\\
B_n^{-}=B\cap\bigcap_{j=1}^{n-1} R^{-j}A\cap R^{-n}(U\cup B),\quad B_n^{+}=B\cap\bigcap_{j=1}^{n-1} R^{-j}C\cap R^{-n}(B\cup W),\\
\alpha_J=\{U_n, W_n, B_1, B_{n+1}^-, B_{n+1}^+: n\in\mathbb N\}.
\end{gather*}
Then $(J,\mathcal B|_J, {\rm Leb}, R_J, \alpha_J)$ is Gibbs-Markov.
\end{example}

\begin{example}[Rational functions on the Riemann sphere, \cite{DenkerUrbanski1991b}]\label{ex:rational}
Let $S^2$ denote the Riemann sphere. A rational function function $T:S^2\to S^2$ has the form $T(z)=\frac{P(z)}{Q(z)}$ when $T$ is restricted to $\mathbb C\subset S^2$ and where $P$ and $Q$ are polynomials $(Q\neq 0)$ of maximal degree $\geqslant 2$. Restrict $T$ to its Julia set $$J(T)=\{z\in S^2: \{T^n\}_{n\in\mathbb N} \text{ is not normal at } z\}.$$ These transformations include important examples for Gibbs-Markov maps. $T$ is called hyperbolic if $J(T)$ does not contain any critical point or rationally indifferent point. In this case $T:J(T)\to J(T)$ is expanding and has a finite Markov partition. 
For general $T: J(T)\to J(T)$ a probabilistic measure $\mu$ on $J(T)$ is called conformal for a continuous potential $\varphi:J(T)\to \mathbb R$ if $$\mu(TA)=\int_A e^{\varphi(y)}d\mu(y)$$
for every measurable $A$ where $T|_A:A\to TA$ is invertible. Equivalently, $\mu$ is characterized by
$$\mathcal L_{\varphi}^*\mu=\mu$$
where $\mathcal L_{\varphi}(f)(z)=\sum_{Ty=z} f(y)e^{-\varphi(y)}$ denotes the transfer operator. In case of a hyperbolic $T$, for any $\varphi$ there is $\lambda>0$ such that there is a $e^{\lambda\varphi}$-conformal measure $m$ by Sullivan's result \cite{Sullivan1983}. The system $(J(T), m, T, \alpha)$ is Gibbs-Markov where $\alpha$ is the finite Markov partition. The transformation $T$ is called parabolic if $J(T)$ does not contain critical points, but has rationally indifferent periodic points. In this case it is known by \cite{DenkerUrbanski1991b} that a conformal measure exists if the potential $\varphi$ satisfies $P(T,\varphi)>\sup\varphi$ where $P(T,\varphi)$ denotes the pressure of $\varphi$. So such measure exists for $\varphi=t\log|T'|$. 
Also $T$ admits a countable Markov partition. 
$(A, m|_A, T_A, \beta)$ is a Gibbs-Markov systems, whenever $A$ belongs to the Markov partition and $\beta$ is the first return time partition on $A$ which is also a (countable) Markov partition.
\end{example}

It is known that a probability-preserving, topologically mixing Gibbs-Markov map $(\Omega,\mathcal{B},\mu,T,\alpha)$ is continued-fraction mixing (exponential $\psi$-mixing), as stated in the next proposition. A Gibbs-Markov map is called topologically mixing if for any $a, b\in\alpha$, there is $n_{a,b}\in\mathbb N$ such that for every $n\geqslant n_{a,b}$, $b\subset T^n a$. The above examples are all topologically mixing. 

\begin{proposition}[{\cite[Corollary 4.7.8]{Aaronson1997}, see also \cite{AaronsonDenkerUrbanski1993}}]\label{prop:cfmixing}
A probability-preserving, topologically mixing Gibbs-Markov map is continued-fraction mixing. That is, there exist constants $K>0$ and $0<\theta<1$ such that for any $n,k\in\mathbb N$, $a\in\alpha_0^{k-1}$ and $B\in\mathcal B$
$$|\mu(a\cap T^{-n-k}B)-\mu(a)\mu(B)|\leqslant K\theta^n\mu(a)\mu(B).$$
\end{proposition}

For a $\psi$-mixing dynamical system whose $\sigma$-algebra is generated by a countable partition, one can hope for a Poisson limit theorem around a non-periodic point, as mentioned in the introduction. For example, the following Poisson limit theorem can be obtained from \cite[Corollary 1]{HaydnPsiloyenis2014} (see also \cite{HirataSaussolVaienti1999, HaydnVaienti2004}). Note that in there the authors also obtain error estimates.

\begin{theorem}\label{thm:poissonap}
Given a probability-preserving, topologically mixing Gibbs-Markov map. Suppose $x$ is a non-periodic point and $A_n(x)\in\alpha_0^{n-1}$ is the cylinder set of length $n$ that contains $x$. Then for any $t>0$ and $k\in\mathbb N$
\begin{equation*}
\lim_{n\rightarrow\infty}\mu\left(\left\{x\in\Omega: \sum_{i=0}^{[t/\mu(A_n(x)]-1}{\bf 1}_{A_n(x)}\circ T^i(x)=k\right\}\right)=\frac{t^k}{k!}e^{-t}.
\end{equation*}
\end{theorem}
However, around a periodic point or when $A_n$ always intersects with its preimages via small iterations, one can see that the error estimates in \cite{HaydnPsiloyenis2014, HaydnYang2016} do not imply a Poisson limit law. Indeed, when $x$ is a periodic point and $A_n=A_n(x)\in\alpha_0^{n-1}$, the limit distribution is compound Poisson as shown in \cite{HaydnVaienti2009}. We note that similar results hold for the induced measure $\mu_{n}=\frac{\mu|_{A_n}}{\mu(A_n)}$ through the relation revealed in \cite{Zweimueller2016} between the statistics of successive return times and the statistics of successive hitting times. 

Proposition \ref{prop:cfmixing} can be proved with the following Renyi's property and transfer operator's spectral gap property (to be discussed in the next section).

\begin{proposition}{\rm (Renyi's property, \cite[Proposition 1.2]{AaronsonDenker2001})}\label{prop:renyi} Given a topologically mixing Gibbs-Markov map. There exists a constant $M>0$ such that for every $n\in\mathbb N, a\in\alpha_0^{n-1}$, a.e. $x\in T^n a$,
\begin{equation*}
M^{-1}\mu(a)\leqslant v'_a(x)\leqslant M\mu(a).
\end{equation*}
\end{proposition}
A simple corollary is useful to estimate short returns.
\begin{corollary}\label{lem:relidp}Given a probability-preserving, topologically mixing Gibbs-Markov map. There exists a constant $M_1>0$ such that for every $n\in\mathbb N, a\in\alpha_0^{n-1}, k\leqslant n$, 
$$\mu(a\cap T^{-k}a)\leqslant M_1\mu(a)^{1+\frac{1}{1+n}}.$$ 
\end{corollary}
\begin{proof}
Let $r=n-[n/k]k$. Let $b\in\alpha_0^{k-1}$ be the cylinder set of length $k$ that contains $a$ as a subset. 
Note that if $a\cap T^{-k}a\neq\emptyset$,
\begin{gather*}
a=b\cap T^{-k} b\cap\cdots\cap T^{-([n/k]-1)k}b\cap T^{-[n/k]k}T^{n-r}a,\\
a\cap T^{-k}a=b\cap T^{-k}a,\\
a\cap T^{-k}a\subset a\cap T^{-(n+k-r)}T^{n-r}a=a\cap T^{-n}(T^{-(k-r)}T^{n-r}a).
\end{gather*}
Then it follows from Renyi's property that there exists a constant $M$ such that
\begin{gather*}
\mu(a)\geqslant M^{-[n/k]}\mu(b)^{[n/k]}\mu(T^{n-r}a),\\
\mu(a\cap T^{-k}a)\leqslant M \mu(b)\mu(a),\\
\mu(a\cap T^{-k}a)\leqslant M \mu(a)\mu(T^{n-r}a).
\end{gather*}
Raise the second inequality to $[n/k]$-th power, multiply with the third inequality, then substitute the first inequality, one gets
\begin{align*}
\mu(a\cap T^{-k}a)^{[n/k]+1}&\leqslant M^{[n/k]+1}\mu(a)^{[n/k]+1}\mu(b)^{[n/k]}\mu(T^{n-r}a)\\
&\leqslant M^{2[n/k]+1}\mu(a)^{[n/k]+2}.
\end{align*}
Therefore
$$\mu(a\cap T^{-k}a)\leqslant M^2 \mu(a)^\frac{2+[n/k]}{1+[n/k]}\leqslant M^2 \mu(a)^{1+\frac{1}{1+n}}.$$
\end{proof}

\section{Transfer operator and perturbation}\label{sec:transfer}
From now on we always assume that $(\Omega, \mathcal B, \mu, T, \alpha)$ is a probability-preserving, topologically mixing Gibbs-Markov system. In this section we summarize some properties of the widely-used transfer operator.
For any partition $\rho$ of $\Omega$, define the H\"older norm subject to $\rho$ of a function $f:\Omega\to \mathbb R$ by
\begin{equation*}
D_\rho f:=\sup\limits_{b\in\rho}\sup\limits_{x, y\in b, x\neq y}\frac{|f(x)-f(y)|}{r(x,y)},
\end{equation*}
where  $\sup$ is taken $\mu$ almost everywhere. 
Denote the usual $L^q$-norm by $\|\cdot\|_q$, $1\leqslant q\leqslant\infty$. Set
\begin{equation*}
\|f\|_{\infty,\rho}:=\|f\|_\infty+D_{\rho} f.
\end{equation*}
 Denote by $L^\infty_\rho$ the set consisting of functions of finite $\|\cdot\|_{\infty,\rho}$-norm. $(L^\infty_\rho, \|\cdot\|_{\infty, \rho})$ is a Banach space.
Recall that $T\alpha\subset \sigma(\alpha)$. For every $n\in\mathbb N$, $T^n(\alpha_0^{n-1})=T\alpha$. Fix a partition $\beta$ such that 
$\sigma(T\alpha)=\sigma(\beta).$
Define the transfer operator $\mathcal{L}:L^1(\mu)\rightarrow L^1(\mu)$ by
 $$  \mathcal{L}(f):=\sum_{b\in\beta}{\mathbf 1}_b\sum_{a\in\alpha,Ta\supset b}v'_a\cdot f\circ v_a.$$
Since $\mu$ is $T$-invariant,
$\mathcal{L} ({\mathbf 1})={\mathbf 1}.$
 $\mathcal L$ satisfies and is uniquely characterized by:
  \begin{equation*}\label{eq:transfer}
   \int_\Omega \mathcal{L}(f)\cdot gd\mu=\int_\Omega f\cdot g\circ Td\mu, \qquad \forall f\in L^1(\mu), g\in L^\infty(\mu).
  \end{equation*}
 We write for simplicity
$$L:=L_\beta^\infty, \qquad \|\cdot\|:=\|\cdot\|_{\infty,\beta}.$$ As no confusion should appear, we use the same notation $\|\cdot\|$ for the operator norm on $L$. On $L$, the transfer operator
$\mathcal L$ satisfies Doeblin-Fortet inequality and spectral gap property as stated in the next proposition.
Denote by $\mathfrak r(\cdot)$ the spectral radius of a linear operator. Let $r\in (0,1)$ be the constant in Definition \ref{def:gibbsmarkov} (4). 
\begin{proposition}\cite[Propositions 1.4, 2.1, Theorem 1.6]{AaronsonDenker2001}\label{prop:df} 
There are constants $C_1, C_2>0$ such that the following statements are true.
For all $f\in L$ and $n\in\mathbb N$,
\begin{equation*}
||\mathcal{L}^n(f)||\leqslant C_1 (r^n D_\beta f+||f||_1).
\end{equation*}
Suppose $\omega:\Omega\rightarrow [0,1]$ satisfies $D_\alpha\omega<\infty$. Let $\mathcal{L}_\omega (f):=\mathcal{L}(\omega\cdot f)$ then 
\begin{equation*}
||\mathcal{L}^n_\omega (f)||\leqslant  (C_1+ C_2 D_\alpha\omega)( r^n D_\beta f+||f||_1).
\end{equation*}Furthermore the essential spectral radius $\mathfrak r_{\rm ess}(\mathcal L)\leqslant r$ and $1$ is the simple, unique maximal eigenvalue of $\mathcal L$.
 One can decompose $\mathcal L$ on L as
 \begin{equation}\label{eq:Lspecgap}
 \mathcal{L}=P+N,
 \end{equation}
where $P(f)=\int_\Omega fd\mu$ is the eigenprojection of $\mathcal L$ with respect to $1$, $P N=N P=0$ and $\mathfrak r(N)<1$.
\end{proposition}

Following \cite{Hirata1993}, one can investigate the statistics of return times by studying the operator  ${\mathcal{L}}_{A}(f):=\mathcal{L}({\mathbf 1_{A^c}}\cdot f)$ as a perturbation of $\mathcal L$. However in our situation when $A_n$ has a constant cylindrical length, this perturbation is not asymptotically small under $\|\cdot\|$-norm. Take the continued fraction map for example. Recall $\alpha=\{\{a_1=n\}:n\in\mathbb N\}$. The Gauss map sends every element of  $\alpha$ onto the whole interval $(0,1)$, so let $\beta=\{(0,1)\}$. Let $A_n=(0, 1/n)\in\sigma(\alpha)$. Then $\|{\mathcal{L}}_{A_n}-\mathcal L\|=\|\mathcal L_{A_n^c}\|$ is of the order of $D_\beta \mathbf 1_{A_n}=r^{-1}$ for all $n$, not tending to $0$. Usual analytic perturbation theorems are not applicable. Instead as mentioned in the introduction, we use Keller and Liverani's perturbation theorem to deal with such perturbations that are asymptotically small under a weaker norm. We restate Keller and Liverani's theorems from \cite{KellerLiverani1999} as follows.

\begin{theorem}\label{thm:weakperturb}
Let $\{\mathcal L_n\}_{n\in\mathbb N_0}$ be a sequence of bounded linear operators on a Banach space $(B, \|\cdot\|)$, with a second (semi-)norm $\frak n(\cdot)\leqslant\|\cdot\|$. Let $0<\rho<1$ and let $\tau:\mathbb N_0\to \mathbb R^+$ be a monotone function with $\lim_{n\to\infty}\tau_n=0$. Suppose that for some constant $C>0$, and for all $n, k\in \mathbb N_0$ and $f\in B$,
\begin{enumerate}[(i)]
\item $\frak n(\mathcal L_n^k)\leqslant C$,
\item $\|\mathcal L_n^k (f)\|\leqslant C (\rho^k\|f\|+ \frak n(f))$,
\item $\frak n((\mathcal L_0-\mathcal L_n)(f))\leqslant \tau_n\|f\|$.
\end{enumerate}
Suppose further that $\mathfrak r_{\rm ess}(\mathcal L_0)\leqslant \rho$ and $\lambda_0$ is the simple, unique maximal eigenvalue of $\mathcal L_0$ with $|\lambda_0|=1$. Decompose $\mathcal L_0=\lambda_0P_0+N_0$ in the form of \eqref{eq:Lspecgap}.
Let $\rho'\in(\rho,1)$ and $\delta>0$ be small such that $\{z\in\mathbb C:|z-\lambda_0|\leqslant\delta\}$ is disjoint from $\{z\in\mathbb C: |z|\leqslant \rho'\}$. Then there exist $n_0\in\mathbb N, \eta\in (0,1), K>0$ such that for every $n\geqslant n_0$, $\mathfrak r_{\rm ess}(\mathcal L_n)\leqslant \rho'$, $\mathcal L_n$ has a simple, unique maximal eigenvalue $\lambda_n$ and one can decompose $\mathcal L_n=\lambda_n P_n+ N_n$ in the form of \eqref{eq:Lspecgap} satisfying additionally
\begin{enumerate}
\item $|\lambda_n-\lambda_0|\leqslant\delta$,
\item $\frak n((P_n-P_0)(f))\leqslant K\tau^\eta_n\|f\|$ for any $f\in B$,
\item $\|P_n (f)\|\leqslant K\frak n(P_n (f))$ for any $f\in B$,
\item $\|N_n^k\|\leqslant K (1-\delta)^k$ for any $k\in\mathbb N$.
\end{enumerate}
\end{theorem}

\section{Main theorem and applications to continued fractions}
Now we prove the main theorem in this note. We consider a sequence of shrinking sets with bounded lengths. If the sequence converges to a point, it may be a periodic point or a compactification point. To some extent, our result complements both the Poisson limit theorem \ref{thm:poissonap} for a sequence of cylinder sets with increasing lengths converging to a non-periodic point and the compound Poisson limit theorem for a sequence of cylinder sets with increasing lengths converging to a periodic point.
\begin{theorem} \label{thm:main}
Given a measure-preserving, topologically mixing Gibbs-Markov map. Assume that a sequence of measurable sets $\{A_n\}_{n\in\mathbb N}$ satisfies the following conditions.
\begin{enumerate}[(a)]
\item For some $m\in\mathbb N$, $A_n\in\sigma(\alpha_0^{m-1})$ for all $n\in\mathbb N$.
\item For all $i\in\mathbb N$, $\displaystyle\lim_{n\to\infty}\mu(A_n\cap T^{-i}A_n)/\mu(A_n)=0.$
\end{enumerate}
If  $n\mu(A_n)$ converges to a constant $t>0$ as $n\to\infty$, then for every $k\in\mathbb N_0$ \begin{equation*}
\lim_{n\rightarrow\infty}\mu\left(\left\{x\in\Omega: \sum_{i=0}^{n-1}{\bf 1}_{A_n}\circ T^i(x)=k\right\}\right)=\frac{t^k}{k!}e^{-t}.
\end{equation*}
\end{theorem}
\begin{proof}
Let $S_{n, k}=\mathbf 1_{A_n}+\ldots+\mathbf 1_{A_n}\circ T^{k-1}$. We will show that the Laplace transform of $S_{n, n}$ converges to the Laplace transform of the Poisson distribution with parameter $t$,
that is, $\int e^{-sS_{n, n}}d\mu\to e^{-t(1-e^{-s})}$ as $n\to\infty$ for every $s\geqslant 0$. Fix an arbitrary $s\geqslant 0$. Define $\mathcal L_n$ by $$\mathcal L_n (f):=\mathcal L(e^{-s \mathbf 1_{A_n}}f),$$ then for all $k\in\mathbb N$ $$\mathcal L_n^k(f)=\mathcal L^k (e^{-s S_{n,k}}f).$$ 
Set $\mathcal L_0:=\mathcal L$. We verify the conditions (i)-(iii) in Theorem \ref{thm:weakperturb} for operators $\{\mathcal L_n\}_{n\in\mathbb N_0}$ on  $(L, \|\cdot\|)$, a second norm $\frak n(\cdot)=\|\cdot\|_1$, and $\tau_n=(1-e^{-s})\mu(A_n)$.

It is clear that $\|\mathcal L_n^k\|_1\leqslant 1$ for all $n, k\in\mathbb N_0$, so condition (i) is satisfied. Because $A_n\in \sigma({\alpha_0^{m-1}})$, $A_n$ is a countable union of cylinder sets in $\alpha_0^{m-1}$. Let $r\in (0,1)$ be the constant in Definition \ref{def:gibbsmarkov} (4). So $D_\alpha \mathbf 1_{A_n}=0$ when $m=1$, and $D_\alpha \mathbf 1_{A_n}\leqslant r^{-{m}}$ when $m\geqslant 2$. Now Proposition \ref{prop:df} implies a uniform Doeblin-Fortet inequality for $\mathcal L_n$: there are constants $C_1, C_2>0$ such that for all $n,k\in\mathbb N_0$ and $f\in L$,
$$\|\mathcal L_n^k (f)\|\leqslant (C_1+C_2r^{-m}) (r^k\|f\|+ \|f\|_1).$$
This inequality verifies condition (ii). Condition (iii) also holds because
\begin{align*}
\|(\mathcal L_0-\mathcal L_n)(f)\|_1&=\|\mathcal L ((\mathbf 1-e^{-s \mathbf 1_{A_n}})\cdot f)\|_1\\
&\leqslant \int_\Omega (\mathbf 1-e^{-s \mathbf 1_{A_n}})|f| d\mu=(1-e^{-s})\int_{A_n} |f|d\mu\\
&\leqslant (1-e^{-s})\|f\|\mu(A_n)=\tau_n\|f\|.  
\end{align*}
In addition $\mathcal L_0$ has a spectral gap as stated in \eqref{eq:Lspecgap}.
Fix $r'\in(r,1)$ and $\delta>0$ such that $\{|z-1|\leqslant \delta\}$ is disjoint from $\{|z|\leqslant r'\}$. It follows from Theorem \ref{thm:weakperturb} that there exist $n_0\in\mathbb N, \eta\in (0,1), K>0$ such that for every $n\geqslant n_0$, $\frak r_{\rm ess}(\mathcal L_n)\leqslant r'$, $\mathcal L_n$ has a simple, unique maximal eigenvalue $\lambda_n$ and one can decompose 
$\mathcal L_n=\lambda_n P_n+N_n$ in the form of \eqref{eq:Lspecgap} satisfying additionally 
\begin{enumerate}
\item $|\lambda_n-1|\leqslant \eta$,
\item $\int |P_n (f)-\int f d\mu| d\mu\leqslant K\mu(A_n)^\eta\|f\|$ for any $f\in L$,
\item $\|P_n (f)\|\leqslant K \int |P_n (f)| d\mu$ for any $f\in L$,
\item $\|N_n^k\|\leqslant K (1-\delta)^k$ for any $k\in\mathbb N$. 
\end{enumerate}
Then one has, since $\mathcal L_n^n(\mathbf 1)=\mathcal L^n (e^{-s S_{n,n}})$, 
$$\int e^{-sS_{n,n}}d\mu= \int \mathcal L_n^n (\mathbf 1)d\mu=\lambda_n^n\int P_n(\mathbf 1)d\mu+\int N_n^n(\mathbf 1)d\mu,$$
hence $$\lim_{n\to\infty}\int e^{-sS_{n,n}}d\mu=\lim_{n\to\infty} \lambda_n^n$$
whenever the right-hand limit exists. The existence of the limit is implied by the following lemma.
\begin{lemma}\label{lem:ev}
$$\lim_{n\to\infty}\frac{1-\lambda_n}{(1-e^{-s})\mu(A_n)}=1.$$
\end{lemma}
\begin{proof}
Suppose $n\geqslant n_0$. Let $f_n$ be an eigenfunction of $\mathcal L_n$ with repect to $\lambda_n$, then $P_n(f_n)=f_n$ and $N_n(f_n)=0$. Note that $\int f_n d\mu\neq 0$, otherwise using the above properties (2) and (3) about $P_n$, $$\|f_n\|=\|P_nf_n\|\leqslant K\int |P_n(f_n)|d\mu=K \int \left|P_n(f_n)-\int f_nd\mu\right|d\mu\leqslant K^2\mu(A_n)^\eta \|f_n\|.$$ This yields a contradiction when $n\to\infty$. 
So we can assume that $\int f_n d\mu=1$. Repeat the previous argument to see that $$(1-K^2\mu(A_n)^\eta)\|f_n\|\leqslant K.$$
Hence we can further assume that $\|f_n\|$ is uniformly bounded, say $\|f_n\|\leqslant C$ for all $n$.
For every $k\in\mathbb N,$
\begin{equation}\label{eq:eigenvk}
1-\lambda_n^k=(1-\lambda_n^k)\int f_n d\mu=\int (f_n-\mathcal L_n^k (f_n)) d\mu =\int (1-e^{-sS_{n,k}})f_n d\mu.
\end{equation}
When $k=1$ this equation says
$$1-\lambda_n=(1-e^{-s})\int_{A_n}f_nd\mu,$$
so
\begin{equation}\label{eq:eigenvlb}
|\lambda_n|=\left|1-(1-e^{-s})\int_{A_n}f_nd\mu\right|\geqslant 1-(1-e^{-s})C\mu(A_n).
\end{equation}
In \eqref{eq:eigenvk} replace $k$ by $k+1$ to get
\begin{align*}
1-\lambda_n^{k+1}&=\int (1-e^{-sS_{n,k+1}})f_n d\mu\\
&=\int_{(T^{-k}A_n)^c} (1-e^{-sS_{n,k}})f_n d\mu + \int _{T^{-k}A_n} (1-e^{-sS_{n,k}})f_n d\mu \\
&\qquad+ \int_{T^{-k}A_n} (e^{-s S_{n,k}}- e^{-s} e^{-sS_{n,k}})f_n d\mu\\
&=\int (1-e^{-s S_{n,k}})f_n d\mu + (1-e^{-s})\int_{T^{-k}A_n}  e^{-s S_{n,k}}f_n d\mu \\
&=1-\lambda_n^k+(1-e^{-s}) \int_{T^{-k}A_n}  e^{-sS_{n,k}}f_n d\mu.
\end{align*}
Therefore,
\begin{align*}
\lambda_n^k(1-\lambda_n)
&=(1-e^{-s})\int  e^{-s S_{n,k}}f_n\cdot  \mathbf 1_{A_n}\circ T^k d\mu\\
&=(1-e^{-s})\int \mathcal L^k (e^{-s S_{n,k}} f_n) \cdot \mathbf 1_{A_n}d\mu\\
&=(1-e^{-s})\int_{A_n}  \left(\int e^{-sS_{n,k}}f_n d\mu+N^k (e^{-sS_{n,k}}f_n)\right)d\mu\\
&=(1-e^{-s})  \left(\mu(A_n)\int \mathcal L_n^k (f_n) d\mu+\int _{A_n}N^k (e^{-sS_{n,k}}f_n)d\mu\right)\\
&=(1-e^{-s})  \left(\mu(A_n) \lambda_n^k+\int_{A_n}N^k (e^{-sS_{n,k}}f_n)d\mu\right).
\end{align*}
We obtain that for every $k\in\mathbb N$
\begin{equation}\label{eq:eigenvNk}\frac{1-\lambda_n}{(1-e^{-s})\mu(A_n)}= 1+\frac{1}{\mu(A_n)\lambda_n^k} \int_{A_n} N^k (e^{-sS_{n,k}}f_n) d\mu.\end{equation}
In view of \eqref{eq:eigenvlb} we  can choose $k=k(n)$ so that $k(n)\to\infty$ as $n\to\infty$ and that $|\lambda_n|^k>\frac13$ for large $n$. 
Moreover choose $\ell=\ell(n)\in\mathbb N<k(n)$ so that $\ell(n)\to\infty$ and $\ell\mu(A_n)\to 0$ when $n\to\infty$. Now the continued-fraction mixing property (Proposition \ref{prop:cfmixing}) and the assumption $A_n\in \sigma(\alpha_0^{m-1})$ imply that there is a constant $C_3>0$ such that
$$\sum_{i=m+1}^\ell \mu(A_n\cap T^{-i}A_n)\leqslant C_3 (\ell-m) \mu(A_n)^2.$$
Together with the assumption (b) $\mu(A_n\cap T^{-i}A_n)=o(\mu(A_n))$ for all $i\in \mathbb N$, we get that
 $$ \lim_{n\to\infty}\frac1{\mu(A_n)}\sum_{i=1}^\ell \mu(A_n\cap T^{-i}A_n)=0.$$
Back to \eqref{eq:eigenvNk}, by the choice of $k$, it suffices to show $$\int_{A_n} N^k (e^{-sS_{n,k}}f_n) d\mu=o(\mu(A_n)).$$  Separate it into two parts by $$e^{-sS_{n,k}}= e^{-sS_{n,k-\ell}}+e^{-sS_{n,k-\ell}}(e^{-s S_{n,\ell}\circ T^{k-\ell}}-1).$$
The first part \begin{align*}\left|\int _{A_n}N^k (e^{-sS_{n,k-\ell}}f_n) d\mu\right|
&\leqslant K (1-\delta)^\ell\mu(A_n) \|N^{k-\ell} (e^{-sS_{n,k-\ell}}f_n) \|\\
&= K(1-\delta)^\ell\mu(A_n) \|\mathcal L^{k-\ell} (e^{-sS_{n,k-\ell}}f_n) -\int e^{-sS_{n,k-\ell}}f_nd\mu\|\\
&=K(1-\delta)^\ell\mu(A_n) \|\mathcal L_n^{k-\ell} (f_n) -\int e^{-sS_{n,k-\ell}}f_nd\mu\|\\
&\leqslant K (1-\delta)^\ell (C_1+C_2r^{-m}+1)C\mu(A_n)=o(\mu(A_n)).
\end{align*}
The second part, letting $g:=e^{-sS_{n,k-\ell}}(e^{-sS_{n,\ell}\circ T^{k-\ell}}-1)$,
\begin{align*} 
&\quad \left|\int_{A_n} N^k (gf_n) d\mu\right|=\left|\int_{A_n}\mathcal L^k(gf_n)d\mu-\mu(A_n)\int gf_n d\mu\right|\\
&=\left|\int gf_n\cdot \mathbf 1_{A_n}\circ T^k d\mu -\mu(A_n)\int gf_n d\mu\right|\\
&\leqslant  2\|f_n\|_\infty \mu(\{S_{n,\ell}\circ T^{k-\ell}\neq 0\}\cap T^{-k}A_n)+2\mu(A_n)\|f_n\|_\infty\mu(\{S_{n,\ell}\circ T^{k-\ell}\neq 0\})\\
&\leqslant 2C \mu\Bigl (T^{-k}A_n\cap\bigcup_{i=k-\ell}^{k-1}T^{-i}A_n\Bigr) +2C\mu(A_n)\mu\Bigl(\bigcup_{i=k-\ell}^{k-1}T^{-i}A_n\Bigr)\\
&\leqslant 2C\sum_{i=1}^\ell \mu(A_n\cap T^{-i}A_n)+2C\ell\mu(A_n)^2=o(\mu(A_n)).
\end{align*}
Hence we get the desired limit.
\end{proof}
Back to the proof of Theorem \ref{thm:main}. As a byproduct of Lemma \ref{lem:ev}, $|\lambda_n|\leqslant 1$ for large $n$. Then, 
$$|\lambda_n^n-(1-(1-e^{-s})\mu(A_n))^n|\leqslant n\cdot |\lambda_n-1+(1-e^{-s})\mu(A_n)|.$$ 
Because $n\mu(A_n)\to t$ and because of the lemma, the right-hand side converges to $0$, then $$\lim_{n\to\infty}\lambda_n^n=\lim_{n\to\infty}(1-(1-e^{-s})\mu(A_n))^n=e^{-t(1-e^{-s})}.$$
\end{proof}

In the following situations, assumption (b) can be easily verified.
\begin{corollary}\label{cor:poissoncor} The Poisson limit theorem holds when
\begin{enumerate}
\item $A_n\in\sigma(\alpha)$ for all $n\in\mathbb N$, or
\item for some $m\in\mathbb N$, $A_n\in\alpha_0^{m-1}$ for all $n\in\mathbb N$.
\end{enumerate}
\end{corollary}
\begin{proof}
In the first situation, assumption (b) can be deduced from continued-fraction mixing. In the second situation, assumption (b) can be deduced from Corollary \ref{lem:relidp}.
\end{proof}

\begin{remark}
Here is an example when $A_n$ satisfies assumption (a) but not assumption (b). Consider the continued fraction map. Let $A_n=[1, n]\cup [n, 1]\in\sigma (\alpha_0^1)$. That is, $A_n=\{x\in (0,1): \text{either } a_1=1, a_2=n, \text{ or } a_1=n, a_2=1.\}$. Then $\mu(A_n)=O(n^{-2})$. But $A_n\cap T^{-1}A_n=[1, n, 1]\cup [n, 1, n]$ has measure of the order $O(n^{-2})$ too.
\end{remark}

\begin{remark}Applying Theorem \ref{thm:weakperturb} to $\tilde{\mathcal L}_{n}(f):=\mathcal L(\mathbf 1_{A_n^c}\cdot f)$,  one can again decompose $\tilde{\mathcal L}_{n}$ likewise with a leading eigenvalue $\tilde{\lambda}_n$. Keller and Liverani's formula in \cite{KellerLiverani2009} for escape rate gives $$\lim_{n\to\infty}\frac{1-\tilde\lambda_n}{\mu(A_n)}=1.$$
This formula also holds when applying Theorem \ref{thm:weakperturb} to more general cases, which require that $\|\mathbf 1_{A_n}\|$ is uniformly bounded for all $n$, like in assumption (a), and that only a small portion of $\{A_n\}$ will return in short times, like in assumption (b). Consequently, it confirms that the distribution of the first return time tends to an exponential distribution. Note that continued-fraction mixing (or $\psi$-mixing) is not necessary in this line of proof, whereas a direct calculation of escape rate assuming $\psi$-mixing but not necessarily spectral gap can be found in \cite{Zhang2016}.

For $\mathcal L_n$ considered in this note (suited to study successive return times), the same formula in \cite{KellerLiverani2009} implies that 
$$\lim_{n\to\infty}\frac{1-\lambda_n}{(1-e^{-s})\mu(A_n)}=1-\sum_{k=0}^\infty q_k,$$
where \begin{align*}
q_k&=\lim_{n\to\infty}\frac{1}{(1-e^{-s})\mu(A_n)}\int(\mathcal L-\mathcal L_n)\mathcal L_n^k(\mathcal L-\mathcal L_n)(\mathbf 1)d\mu\\
&=\lim_{n\to\infty}\frac{1}{\mu(A_n)}\int e^{-sS_{n,k}}\cdot (2\cdot \mathbf 1_{A_n}- \mathbf 1_{T^{-1}A_n})\circ T^k d\mu.
\end{align*}
It is likely that one can go on to obtain the same Poisson limit theorem by arguments similar to the second half of the proof of Lemma \ref{lem:ev}.
\end{remark}

Lastly, we state some applications of Theorem \ref{thm:main} to continued fractions. Let $((0,1), \mu, T, \alpha)$ be the continued fraction map, where $d\mu=\frac{1}{\log2}\frac{1}{1+x} dx$, $T(x)=\{1/x\}$, and $\alpha=\{\{a_1=n\}: n\in\mathbb N\}$. As seen in Example \ref{ex:mip}, it is a probability-preserving, topologically mixing Gibbs-Markov map. 
Fix any real number $\theta>0$. 
By applying Corollary \ref{cor:poissoncor} (1) to $A_n=(0,\frac{1}{[\theta n]+1})\in\sigma(\alpha)$, we recover Doeblin's theorem in the introduction. Note that as mentioned in \cite{Iosifescu1977}  the Gauss measure can be replaced by the Lebesgue measure, because $|\mu(A)- {\rm Leb}(A)|\leqslant \epsilon_m$ for $A\in\sigma(\alpha_m^\infty)$, where $\epsilon_m\to 0$ as $m\to\infty$.

Choosing $A_n=\bigcup_{a_1,\ldots,a_m\geqslant [(\theta n)^{1/m}]}[a_1, \cdots, a_m]\in\sigma(\alpha_0^{m-1}),$ one can show that $A_n$ satisfies assumption (b) in Theorem \ref{thm:main}. Then an extension of Doeblin's theorem follows.
\begin{corollary} For every $m, k\in\mathbb N_0$
\begin{multline*}\lim_{n\to\infty} {\rm Leb}\{x\in(0,1): \text{there are exactly $k$ m-tuples } a_ia_{i+1}a_{i+m-1}
\text{ with all }\\ a_i,\ldots, a_{i+m-1}> (\theta n)^{1/m}, 1\leqslant i\leqslant n\}=\frac{1}{(\theta \log 2)^k}\frac1{k!} e^{-\frac1{\theta \log 2}}.
\end{multline*}
\end{corollary}
 
Choosing $A_n\in\alpha_0^{m-1}$ such that $n\mu(A_n)$ converges, one can find by Corollary \ref{cor:poissoncor} (2) limit laws for the number of occurrences of pattern $A_n$. For example, if one chooses $A_n=[[n^{1/4}], [n^{1/4}]]\in\alpha_0^2$ then
\begin{corollary} For every $k\in\mathbb N_0$
\begin{multline*}
\lim_{n\to\infty}{\rm Leb}\{x\in (0,1): \text{there are $k$ occurrences of $[n^{1/4}][n^{1/4}]$ in $[a_1, \ldots, a_{n+1}]$}\}\\=\frac{1}{(\log 2)^k}\frac1{k!} e^{-\frac1{\log 2}}.
\end{multline*}
\end{corollary}

\section*{Acknowledgment}
Major part of the work appeared in Chapter 3 of the author's thesis \cite{Zhang2015}. He thanks his advisor Prof. Manfred Denker for generous help. The author was supported by FAPESP grant \#2018/15088-4.
\bibliographystyle{plain}

\end{document}